\documentclass[a4paper]{article}
\usepackage[english]{babel}
\usepackage[cp1251]{inputenc}
\usepackage[colorlinks,linkcolor=blue,citecolor=blue]{hyperref}
\usepackage{latexsym,amsfonts,amssymb,amsthm,amsmath,amsbsy,graphicx}
\newtheorem{theorem}{Theorem}%[section]
\newtheorem{lema}{Lemma}%[section]
\newtheorem{corollary}{Corollary}%[section]

\newcommand{\ad}[1]{\stackrel{{\scriptscriptstyle\textsf{v}}}{#1}}
\newcommand{\Ps}{\mathbf{P}_s^{-1}}

\newcommand{\ps}{\ad{\mathbf{p}}_-(s)}
\newcommand{\qs}{\ad{\mathbf{q}}_-(s)}
\newcommand{\Rs}{\ad{\mathbf{R}}_-(s)}
\newcommand{\Rcs}{\ad{\mathbf{R}}_c(s)}

\newcommand{\Ro}{\mathbf{p}^-_*(0)}

\begin{document}
\thispagestyle{myheadings} \markboth{Overshoot functionals for
almost semi-continuous processes}{Theor. Probability and Math.
Statist. No. 76 (2008), 49-57.}
\bigskip
 {\noindent \bf \large   OVERSHOOT FUNCTIONALS FOR ALMOST
 SEMI-CONTINUOUS\\ PROCESSES DEFINED ON A MARKOV CHAIN}
 \footnotemark[1] \footnotetext{This is an
electronic reprint of the original article published in Theor.
Probability and Math. Statist. No. 76 (2008), 49-57. This reprint
differs from the original in pagination and typographic detail.}

\bigskip
\bigskip
 { \bf  Ievgen~Karnaukh
 \footnote{Department of Probability Theory and Mathematical Statistics,
    Faculty of Mathematics and Mechanics,
    Taras Shevchenko Kyiv National University,
    Glushkova str., 6,
    Kyiv 03127,
    Ukraine.
 \href{mailto:kveugene@mail.ru}{kveugene@mail.ru}
 }}\hskip 10 cm UDC 519.21
\begin{center}
\bigskip
\begin{quotation}
\noindent  {\small In the article the distributions of overjump
functionals for almost semi-continuous processes on a finite
irreducible Markov chain are considered.}
\end{quotation}
\end{center}
\bigskip

The distribution of extrema and overjump functionals for the
semi-continuous processes (processes that intersect positive or
negative level continuously) on a Markov chain were considered by
many authors( for instance, see~\cite{Gusak2}~-~\cite{Gusak1}). In
the paper~\cite{Gusak5} the distribution of extrema for almost
semi-continuous processes were treated (the processes that intersect
 positive or negative level by exponentially distributed jumps).
Under some conditions these processes we can consider as surplus
risk processes with stochastic premium function in a Markov
environment. In the article the distribution of some overjump
functionals for the almost semi-continuous processes defined on a
Markov chain are considered.

Consider a two-dimensional Markov process:
$$Z(t)=\{\xi(t),x(t)\}\quad t\ge 0,$$
\noindent where $x(t)$ is a finite irreducible nonperiodic Markov
chain with the set of states ${E}'=\{1, \ldots, m\}$ and the matrix
of transition probabilities
\begin{equation*}
\mathbf{P}(t)=e^{t \mathbf{Q}},\qquad t\ge0,\qquad
\mathbf{Q}=\mathbf{N}(\mathbf{P}-\mathbf{I}),
\end{equation*}
where $\mathbf{N}=||\delta_{kr}\,\nu_k||^{m}_{k,r=1}$, $\nu_k$ are
the parameters of exponentially distributed random variables
$\zeta_k$ (the sojourn time of $x(t)$ in the state $k$),
$\mathbf{P}=\|p_{kr}\|$ is the matrix of transition probabilities of
the imbedded chain, $\boldsymbol{\pi}=(\pi_1,\ldots,\pi_m)$ is the
stationary distribution. $\xi(t)$ is a process with stationary
conditionally independent increments for fixed values of $x(t)$
(see~\cite{Gusak2}).
\par The evolution of the process $Z(t)$ is described by the matrix
characteristic function:
$$\boldsymbol{ \Phi}_t(\alpha)=
\|\mathrm{E}[e^{\imath \alpha (\xi(t+u)-\xi(u))},x(t+u)=
r/x(u)=k]\|,\qquad u\geq 0,$$ which we can represent as follows
$$\boldsymbol{ \Phi}_t(\alpha)=\mathbf{E}e^{\imath \alpha \xi(t)}=
e^{ t \boldsymbol{\Psi}(\alpha)},\quad
\boldsymbol{\Psi}(0)=\mathbf{Q}.$$
In what follows, we consider processes that have cumulant
\begin{equation}\label{eq1}
  \boldsymbol{\Psi }(\alpha )=\int_0^{\infty}
  \left(e^{\imath\alpha x}-1 \right)d\mathbf{K}_0(x)
  +\boldsymbol{\Lambda }\mathbf{F}_0(0)
  \left(\mathbf{C}\left(\mathbf{C}+\imath\alpha \mathbf{I}\right)^{-1}
  -\mathbf{I} \right)+\mathbf{Q},
\end{equation}
where
$d\mathbf{K}_0(x)=\mathbf{N}d\mathbf{F}(x)+\boldsymbol{\Pi}(dx)$,
$$\mathbf{F}(x)=\|\mathrm{P}\{\chi_{kr}<x;x(\zeta _1)=r/x(0)=k\}\|,$$
 $\chi _{kr}$
are the jumps of $\xi (t)$ at the time of transition of $x(t)$ from
the state $k$ to the state $r$.
$$\boldsymbol{\Pi}(dx)=\boldsymbol{\Lambda }d\mathbf{F}_0(x),\;\;
\mathbf{F}_0(x)=\|\delta _{kr}F^0_k(x)\|,$$ $F^0_k(x)$ are the
distribution functions of the jumps of $\xi (t)$ if
 $x(t)=k$,
$\boldsymbol{\Lambda}=\|\delta _{kr}\lambda _k\|$, $\lambda _k$ are
the parameters of exponentially distributed random variables
$\zeta'_k$ (the time interval between two neighboring jumps of $\xi
(t) $ if $x(t)=k$). $\mathbf{C}=\left |\left |\delta
_{kr}c_{k}\right |\right |$, where $c_k$ are the parameters of
exponentially distributed positive jumps of  $\xi (t)$ if $x(t)=k$.
The process $Z(t)$ with this cumulant is the almost
lower-semicontinuous process defined in~\cite[p.43]{Gusak5}.
\par Let $\theta _s$ denote an exponentially distributed random variable with
parameter $s>0$ $(\mathrm{P}\{\theta _s>t\}=e^{-s t}, t\ge 0 )$,
independent of $Z(t)$. In this case, we rewrite the characteristic
function of $\xi (\theta _s)$ as follows
\begin{equation}\label{eq2}
\boldsymbol{\Phi }(s,\alpha )=\mathbf{E}e^{\imath \alpha \xi(\theta
_s)}= s\int_0^\infty e^{-st}\boldsymbol{\Phi}_t(\alpha )dt
=s\left(s\mathbf{I}-\boldsymbol{ \Psi }(\alpha )\right)^{-1}.
\end{equation}
\par Denote the next functionals for $\xi (t)$:
\begin{gather*}
\begin{aligned}
\xi ^{\pm }(t)=\sup\limits_{0\leq u\leq t}(\inf)\xi (u),&\quad  \xi
^{\pm }=\sup
   \limits_{0\leq u\leq \infty}(\inf)\xi (u),\\
 \overline{\xi }(t)=\xi (t)-\xi ^+(t),&\quad \ad{\xi}(t)=\xi^- (t)-\xi(t);\\
\tau^+(x)=\inf\{t:\xi (t)>x\},  &\quad  \gamma ^+(x)=\xi (\tau ^+(x))-x; \\
 \gamma _+(x)=x-\xi (\tau ^+(x)-0),  &\quad  \gamma _x^+=\gamma ^+(x)+\gamma_+(x) .
\end{aligned}\;\;(x\geq 0)
\end{gather*}

The distributions of $\xi^{\pm}(\theta_s), \overline{\xi}(\theta_s),
\ad{\xi}(\theta_s)$ were concretely defined in~\cite{Gusak5}. The
aim of our article is to find joint moment generating function of
overjump functionals for almost lower semi-continuous processes and
the moment generating functions for pairs $\left\{\tau ^+(x),\gamma
^+(x)\right\}$, $\left\{\tau ^+(x)\right.$, $\left.\gamma
_+(x)\right\}$, $\left\{\tau ^+(x), \gamma ^+_x\right\}$.
\par Denote
\begin{gather*}
  \mathbf{V}(s,x,u,v,\mu )=
  \mathbf{E}\left [e^{-s\tau ^+(x)-u\gamma ^+(x)-v\gamma _+(x)-\mu \gamma ^+_x},\; \tau
     ^+(x)<\infty \right ],\\
  \mathbf{W}(x,u,v,\mu )=\int_x^{\infty }e^{(u-v)x-(u+\mu )z}d\mathbf{K}_0(z),
     \;\overline{\mathbf{K}}_0(x)=\mathbf{W}(x,0,0,0),\\
     \mathbf{P}_s=s\int_{0}^{\infty}e^{-st}\mathbf{P}(t)dt=s\left(s\mathbf{I}-\mathbf{Q}
     \right)^{-1},\\
           \mathbf{P}_+(s,x)=\mathbf{P}\left\{\xi ^+(\theta_s)<x\right\}, x>0;
\mathbf{P}^-(s,x)=\mathbf{P}\left\{\overline{\xi}(\theta_s)<x\right\},
x<0;\\
\tilde{\mathbf{P}}^0(s)=\mathbf{P}\left\{\xi(\theta_s)=0\right\},\;
 \mathbf{p}_{\pm}(s)=\mathbf{P}\left\{\xi
^{\pm}(\theta_s)=0\right\},\; \Rs=\mathbf{P}^{-1}_s\ps,\\
\ps=\mathbf{P}\left\{\overline{\xi}
(\theta_s)=0\right\}, \qs=\mathbf{P}_s-\ps,\;\Rcs=\Rs\mathbf{C},\\
  \mathbf{G}_+(s,x,u,v,\mu )=\int_{-\infty }^0 d\mathbf{P}^-(s,y)\mathbf{W}(x-y,u,v,\mu ).
\end{gather*}
\begin{lema}
  For process $Z(t)$ with cumulant{~\rm{(\ref{eq1})}} the next
  relation holds
\begin{equation}\label{eq3}
  s\mathbf{V}(s,x,u,v,\mu )=\int_0^x d\mathbf{P}_+(s,y)\Ps
  \mathbf{G}_+(s,x-y,u,v,\mu ),\quad x>0,
\end{equation}
where for $u-v\notin \sigma \left(\mathbf{P}^{-1}_s\ps\mathbf{C}
\right)$ (spectrum of matrix)
\begin{multline}\label{eq4}
\mathbf{G}_+(s,x,u,v,\mu )=\, \ps\int_x^{\infty }e^{(u-v)x-(u+\mu )z}d\mathbf{K}_0(z)-\\
  \shoveleft-\ps\mathbf{C} \left(\mathbf{P}^{-1}_s\ps\mathbf{C}
  -(u-v)\mathbf{I}\right)^{-1}e^{-(v+\mu )x}\cdot\int_0^{\infty}
  \biggl[  (u+\mu )e^{-(u+\mu )z}-\\
  \left.-\left(\mathbf{P}^{-1}_s\ps\mathbf{C}+(\mu+v)\mathbf{I}\right)
  e^{-\left(\mathbf{P}^{-1}_s\ps\mathbf{C}+(v+\mu
  )\mathbf{I}\right)z}\right]\Ps\qs\overline{\mathbf{K}}_0(x+z)dz.
\end{multline}
\end{lema}
\begin{proof} Taking into account the condition of semi-continuity, formula~\eqref{eq3} follows
from~\cite[Corollary~3.4]{Gusak2}. By results
of~\cite{Gusak5}(see~Remark 1) the distribution of $\overline{\xi
}(\theta_s)$ is determined by the relation
\begin{equation*}%\label{eq2a}
\mathbf{P}^{-}(s,x)=e^{\ps\mathbf{C}\Ps x}\qs,\; x<0,
\end{equation*}
then
\begin{multline}\label{eq5}
  \mathbf{G}(s,x,u,v,\mu )=\int_{-\infty }^0d\mathbf{P}^-(s,y)
  \mathbf{W}(x-y,u,v,\mu )=\\
=\, \ps\mathbf{W}(x,u,v,\mu )\\+\ps\mathbf{C}\int_x^{\infty
}e^{\Ps\ps\mathbf{C}(x-y)}\Ps\qs\mathbf{W}(y,u,v,\mu )dy.
\end{multline}
Taking into account the definition of $\mathbf{W}(s,x,u,v,\mu)$ we
deduce
\begin{multline*}
\int_x^{\infty }e^{\mathbf{R}^-_*(s)(x-y)}\Ps\qs\mathbf{W}(y,u,v,\mu
)dy=\\
= -\left(\Ps\ps\mathbf{C}-(u-v)\mathbf{I}\right)^{-1}e^{-(v+\mu
)x}\times\\
%\begin{aligned}
\;\biggl[\left(\Ps\ps\mathbf{C}+(\mu+v)\mathbf{I}\right)
\int_0^{\infty}e^{-\left(\Ps\ps\mathbf{C}+(v+\mu
  )\mathbf{I}\right)z}\Ps\qs\overline{\mathbf{K}}_0(x+z)dz\\
 -(u+\mu
)\int_0^{\infty}e^{-(u+\mu)z}\Ps\qs\overline{\mathbf{K}}_0(x+z)dz\biggr],
%\end{aligned}
\end{multline*}
and from{~\rm{(\ref{eq5})}} we obtain{~\rm{(\ref{eq4})}}.
\end{proof}
Note, that $$\lim_{x\rightarrow -\infty}\mathbf{P}^-(s,x)=
\mathbf{P}\left\{\overline{\xi}(\theta_s)<-\infty\right\}=0.$$

Then from the next formula
$$
  \mathbf{P}^-\left(s,x \right)=e^{\ps\mathbf{C}\Ps x}\qs=
\mathbf{P}_s\Ps e^{\ps\mathbf{C}\Ps x}\mathbf{P}_s\Ps\qs=
\mathbf{P}_s e^{\Rcs x}\left(\mathbf{I}-\Rs \right)
$$
 we get, that the spectrum of matrix $\Rcs$
$\left(\sigma \left( \Rcs\right) \right)$ consists of strictly
positive elements.

Denote $$\gamma _1(x)=\gamma ^+(x),\,\gamma _2(x)=\gamma
_+(x),\,\gamma _3(x)=\gamma ^+_x.$$ Substituting $v=\mu =0, (u\notin
\sigma (\Rcs))$, $u=\mu =0$ and $v=u=0$ in~\eqref{eq3} we find that
\begin{equation*}%\label{eq6a}
\mathbf{E}\,\left [ e^{-s\tau ^+(x)-u\gamma _i(x)}, \tau
^+(x)<\infty\right
]=s^{-1}\int_{0}^{x}d\mathbf{P}_+(s,y)\Ps\mathbf{G}_i(s,x-y,u),\;
i=\overline{1,3};
\end{equation*}
\begin{multline*}
\mathbf{G}_1(s,x,u)=
\ps\int_{x}^{\infty}e^{u(x-z)}d\mathbf{K}_0(z)-\\
  -\ps\mathbf{C}\left(\Rcs-u\mathbf{I} \right)^{-1}
 \int_{0}^{\infty}\left [  ue^{-uz}-\Rcs e^{-\Rcs z}\right
 ]\Ps\qs\overline{\mathbf{K}}_0(x+z)dz;\\
\end{multline*}
%\begin{multline*}
$$ \mathbf{G}_2(s,x,v)=\ps e^{-v
  x}\overline{\mathbf{K}}_0(x)+\ps
  \mathbf{C}e^{-vx}\!\!\int_{0}^{\infty}\!\!
  e^{-\left(\Rcs+v \mathbf{I} \right)z}
  \Ps \qs\overline{\mathbf{K}}_0(x+z)dz;$$
%\end{multline*}
\begin{multline*}
\mathbf{G}_3(s,x,\mu )=\ps\int_{x}^{\infty}e^{-\mu
 z}d\mathbf{K}_0(z)-\\
  -e^{-\mu x}\int_{0}^{\infty}\left [  \mu e^{-\mu z}-\left(\Rcs+\mu \mathbf{I} \right)
  e^{-\left(\Rcs+\mu \mathbf{I} \right)z}\right ]\Ps
  \qs\overline{\mathbf{K}}_0(x+z)dz.
\end{multline*}
After inverting with respect to $u$ we deduce
$$\mathbf{E}\,\left [ e^{-s\tau ^+(x)},\gamma _i(x)\in dz, \tau ^+(x)<\infty\right ]
=s^{-1}\int_{0}^{x}d\mathbf{P}_+(s,y) \Rs
d_{z}\mathbf{g}^*_i(s,x-y,z),
$$
where
$$d_{z}\mathbf{g}^*_i(s,x,z)=d_{z}\mathbf{w}^*_i(x,z)
+\mathbf{C}\int_{x}^{\infty}e^{\Rcs(x-y)}
\left(\mathbf{I}-\Rs\right)d_{z}\mathbf{w}^*_i(y,z)dy,$$
$$d_{z}\mathbf{w}^*_1(x,z)=d_{z}\mathbf{K}_0(x+z), d_{z}\mathbf{w}^*_2(x,z)
=d_{z}I\left\{z>x\right\}
\overline{\mathbf{K}}_0(x),$$
$$
d_{z}\mathbf{w}^*_3(x,z)=I\left\{z\geq x\right\}d\mathbf{K}_0(z).
$$
For the case $x=0$ the next assertion is true.
\begin{theorem}
For process $Z(t)$ with cumulant{~\rm{(\ref{eq1})}}, if $z>0 $ we
have:
\begin{multline}\label{eq19}
\mathbf{E}\left [e^{-s\tau ^+(0)},\,\gamma^+(0)>z,\,\tau
  ^+(0)<\infty \right ]=\\
 \quad \quad =s^{-1}\widetilde{\mathbf{P}}^0(s)\left(\overline{\mathbf{K}}_0(z)
 +\mathbf{C}\int_z^{\infty }e^{(z-y)\Rcs}\Ps\qs\overline{\mathbf{K}}_0(y)dy\right),\\
 \mathbf{E}\left [e^{-s\tau ^+(0)},\,\gamma _+(0)>z,\,\tau
  ^+(0)<\infty \right ]=s^{-1}\widetilde{\mathbf{P}}^0(s)\mathbf{C}\int_z^{\infty }e^{-y\Rcs}
  \Ps\qs\overline{\mathbf{K}}_0(y)dy,\\
\shoveleft { \mathbf{E}\left [e^{-s\tau ^+(0)},\,\gamma ^+_0>z,\,\tau
  ^+(0)<\infty \right ]=}\\
\quad \quad
=s^{-1}\widetilde{\mathbf{P}}^0(s)\left(\overline{\mathbf{K}}_0(z)
+\mathbf{C}\int_z^{\infty }\int_{0}^{y}e^{-x\Rcs}dx\Ps\qs
d\mathbf{K}_0(y)\right).
\end{multline}
\end{theorem}
\begin{proof}
From Eq.~\eqref{eq3} it follows that
\begin{equation}\label{eq7a}
  \mathbf{V}\left(s,x,u,v,\mu
  \right)=\overline{\mathbf{P}}_+(s,x)\Ps+
  s^{-1}\int_{0}^{x}d\mathbf{P}_+(s,y)\Ps
  \overline{\mathbf{G}}_+(s,x-y,u,v,\mu ),\;x>0,
\end{equation}
where
$$\overline{\mathbf{G}}_+(s,x,u,v,\mu )=
{\mathbf{G}}_+(s,x,u,v,\mu )-{\mathbf{G}}_+(s,x,0,0,0 ).$$ Taking
into account that for $k=\overline{1,3}$:
\begin{multline*}%\label{eq14}
   \mathbf{V}_k(s,x,u)=\mathbf{E}\left [  e^{-s\tau ^+(x)-u\gamma _k(x)},
   \tau ^+(x)<\infty\right ]=\\
   =\overline{\mathbf{P}}_+(s,x)\Ps
  -u\int_0^{\infty}\!\!\!\!e^{-u z}
  \mathbf{E}\!\!\left [e^{-s\tau ^+(x)},\,\gamma _k(x)>z,\,\tau
  ^+(x)<\infty\! \right ]dz,
\end{multline*}
from~\eqref{eq7a} we obtain
\begin{equation}\label{eq15}
\int_0^{\infty}e^{-u z}\mathbf{E}\left [e^{-s\tau ^+(x)},\,\gamma
_k(x)>z,\,\tau ^+(x)<\infty \right ]dz=-\frac{1}{s u}\int_{-0}^x
d\mathbf{P}_+(s,y)\Ps\overline{\mathbf{G}}_{k}(s,x-y,u),
\end{equation}
where
\begin{gather}
\overline{\mathbf{G}}_k(s,x,u)=\;\ps
\overline{\mathbf{W}}_k(x,u)+\int_{-\infty }^{0-}
d\mathbf{P}^-(s,y)\overline{\mathbf{W}}_k(x-y,u),\label{eq16}\\
\overline{\mathbf{W}}_1(x,u)=\int_x^{\infty
}\left(e^{u(x-z)}-\mathbf{I}\right)d\mathbf{K}_0(z),
\overline{\mathbf{W}}_2(x,u)=\left(e^{-u x}-\mathbf{I}\right)
\overline{\mathbf{K}}_0(x),\label{eq17}\\
\overline{\mathbf{W}}_3(x,u)=\int_x^{\infty }\left(e^{-u
z}-\mathbf{I}\right)d\mathbf{K}_0(z).\notag
\end{gather}
After the limit passage as $x\rightarrow 0$ from{~\rm{(\ref{eq15})}}
we get
\begin{equation}\label{eq18}
\int_0^{\infty}e^{-u z}\mathbf{E}\left [e^{-s\tau ^+(0)},\,\gamma
_k(0)>z,\,\tau ^+(0)<\infty \right ]dz=-\frac{1}{s
u}\mathbf{p}_+(s)\Ps\overline{\mathbf{G}}_k(s,0,u).
\end{equation}

Substituting $x=0$ in{~\rm{(\ref{eq16})}} and using
{~\rm{(\ref{eq17})}} we get
\begin{multline*}
\overline{\mathbf{G}}_1(s,0,u)=-u \ps\biggl(\int_0^{\infty}e^{-u
z}\overline{\mathbf{K}}_0(z)dz+\\+\mathbf{C}\int_0^{\infty}e^{-u
z}\int_z^{\infty }e^{(z-y)\Rcs}\Ps\qs\overline{\mathbf{K}}_0(y)dy dz\biggr),\\
\shoveleft{\overline{\mathbf{G}}_2(s,0,u)=-u
\ps\mathbf{C}\int_0^{\infty}e^{-u
z}\int_z^{\infty }e^{-y\Rcs}\Ps\qs\overline{\mathbf{K}}_0(y)dy dz,}\\
\shoveleft{\overline{\mathbf{G}}_3(s,0,u)=
-u\ps\biggl(\int_0^{\infty}e^{-u z}\overline{\mathbf{K}}_0(z)dz}+\\
\qquad\qquad+\mathbf{C}\int_0^{\infty}e^{-u
z}\int_z^{\infty }\int_{0}^{y}e^{-x\Rcs}dx\Ps\qs
d\mathbf{K}_0(y)dz\biggl).
\end{multline*}
Substituting $\overline{\mathbf{G}}_k(s,0,u)$ in{~\rm{(\ref{eq18})}}
and using the relation $\mathbf{p}_+(s)\Ps\mathbf{p}^-(s)
=\widetilde{\mathbf{P}}^0(s)$(see~\cite[p.47]{Gusak5}
with~\cite[Remark 1]{Gusak5}),
 after inversion with respect to $u$ we receive{~\rm{(\ref{eq19})}}.
\end{proof}
Consider some corollaries of Theorem 1 and results
of~\cite{Asmussen}, namely, the analog of the inverted
Pollaczeck-Khinchine formula and two-sided Lundberg's inequality.
Assume hereinafter that  $\chi _{kr}=0$, $k,r=\overline{1,m}$. The
almost semi-continuous processes that satisfy such conditions we can
consider as surplus risk processes with stochastic premium function
in a Markov environment.

 Let $\zeta ^*$ is the moment of the first
jump of $\xi (t)$. We have the next stochastic relations
(see~\cite[p.42]{Gusak2})
\begin{equation*}
  \zeta ^*_{kr}\dot{=}
  \begin{cases}
    \zeta _k +\zeta ^*_{jr}& \zeta '_k>\zeta _k, x(\zeta _k)=j; \\
    \zeta '_k & \zeta '_k<\zeta _k,
  \end{cases}
\end{equation*}
where indices $kr$ means that $x(\zeta ^*)=r$, $x(0)=k$
$(k,r=\overline{1,m})$. Taking into account the definition of
$Z(t)$, these relations yield (see~\cite[p.64]{Gusak2})
\begin{multline*}
  \mathrm{E}e^{-s\zeta ^*_{kr}}=\mathrm{E}\,\left[e^{-s\zeta ^*}, x(\zeta ^*)=r/x(0)=k
  \right]=\\
  =\mathrm{E}\,\left[ e^{-s\zeta '_k},\zeta '_k<\zeta _k\right]\delta
  _{kr}+\sum_{j=1}^{m}\mathrm{E}\,
  \left[e^{-s\zeta ^*_{jr}+\zeta _k},\zeta '_k>\zeta _k, x(\zeta _k)=j
  \right]\\
  =\int_{0}^{\infty}\lambda _ke^{-sy}e^{-\lambda _k y}e^{-\nu _k y}dy\delta
  _{kr}+\sum_{j=1}^{m}\int_{0}^{\infty}e^{-sy}\nu _ke^{-\nu _k y}e^{-\lambda _k
  y}\mathrm{E}e^{-s\zeta ^*_{jr}}p_{kj}dy\\
  =\lambda _k\left(s+\lambda _k+\nu _k \right)^{-1}\delta _{kr}+\sum_{j=1}^{m}\nu
  _k(s+\nu _k+\lambda _k)^{-1}p_{kj}\mathrm{E}e^{-s\zeta ^*_{jr}},
\end{multline*}
or in a matrix form
$$\mathbf{E}e^{-s\zeta ^*}=\boldsymbol{\Lambda}
\left(s\mathbf{I}+\boldsymbol{\Lambda}+\mathbf{N} \right)^{-1}
+\left(s\mathbf{I}+\boldsymbol{\Lambda}+\mathbf{N}
\right)^{-1}\mathbf{N}\mathbf{P}\,\mathbf{E}e^{-s\zeta ^*}.
$$
Whence
\begin{equation*}%\label{eq19_1}
  \mathbf{E}e^{-s\zeta ^*}=\left(s\mathbf{I}+\boldsymbol{\Lambda}-\mathbf{Q} \right)^{-1}
  \boldsymbol{\Lambda}.
\end{equation*}
Taking into account that $
  \widetilde{\mathbf{P}}^0(s)=
  \left(\mathbf{I}-\mathbf{E}\,e^{-s\zeta ^*} \right)\mathbf{P}_s,
$ we get
\begin{equation*}%\label{eq19a}
  \lim_{s\rightarrow 0}s^{-1}\widetilde{\mathbf{P}}^0(s)=\left(
  \boldsymbol{\Lambda}-\mathbf{Q}\right)^{-1}=\|\mathrm{P}\left\{x(\zeta
  ^*)=r/x(0)=k\right\}\|\boldsymbol{\Lambda}^{-1}.
\end{equation*}
Denote
$$
  m_1^0=\sum_{k=1}^m \pi_k\int_{R}x \lambda _kF^0_k(dx).
$$
and assume that $m_1^0<\infty$.
\begin{corollary}
For $m_1^0<0$
\begin{equation}\label{eq20}
  1-\psi _i(u)=\mathrm{P}\left\{\xi ^+\leq u/x(0)=i\right\}=
  \mathrm{P}_i\left\{\xi ^+\leq u\right\}=\mathbf{e}'_i
  \sum_{n=0}^{\infty }\mathbf{G}^{*n}_+(u)
  \left(\mathbf{I}-\left |\left |\mathbf{G}\right |\right |\right)\mathbf{e},
\end{equation}
\begin{multline}\label{eq21}
  \mathbf{G}_+(u)=\int_{0}^{u}\mathbf{G}_+(dy)=
  \int_{0}^{u}\mathbf{P}\left\{\gamma ^+(0)\in dy, \tau ^+(0)<\infty\right\}=\\
  =\left(\boldsymbol{\Lambda}-\mathbf{Q}\right)^{-1}\int_{0}^{u}
  \left(\boldsymbol{\Lambda}{\mathbf{F}}_0(dy)
  +\mathbf{C}\int_{-\infty }^0 e^{\Ro\mathbf{C}x}\left(\mathbf{I}-\Ro\right)
  \boldsymbol{\Lambda}
  {\mathbf{F}}_0(dy-x)dx\right),
\end{multline}
$\mathbf{G}_+^{*n}(u)$ -- $n$ - fold convolution of
$\mathbf{G}_+(u)$ with itself, $ \left |\left |\mathbf{G}\right
|\right |=\int_0^{\infty}\mathbf{G}_+(dx)$, $\mathbf{e}
=\left(1,\ldots,1\right)'$(column vector),
$\mathbf{e}'_i=\left(0,\ldots,\stackrel{i}{1},\ldots,0\right)$(row
vector).
\end{corollary}
\begin{proof}
Formula{~\rm{(\ref{eq20})}} were obtained in~\cite{Asmussen}(see the
proof of Proposition 2.2) for the processes, that intersect negative
level continuously. However, the proof is also true for the
processes for which pair $\left\{\tau ^+(0),\gamma ^+(0)\right\}$
has nondegenerate distribution. Since $Z(t)$ is the stepwise
process, then formula~\eqref{eq20} holds for our process.
Formula{~\rm{(\ref{eq21})}} follows from the first formula
in{~\rm{(\ref{eq19})}}.
\end{proof}
Let $k(r)$ be the real eigenvalue with maximal absolute value
(Perron's root) of the matrix $\mathbf{K}(r)=\boldsymbol{\Psi
}(-\imath r )$. Suppose that a solution $\gamma>0 $ of the equation
$k(r)=0$ exists and $\boldsymbol{\nu }=\left(\nu _1,\ldots,\nu
_m\right)$, $\mathbf{h}=\left(h_1,\ldots,h_m\right)'$ are
corresponding left and right eigenvectors of the matrix
$\mathbf{K}(\gamma )$. We assume that vectors $\boldsymbol{\nu
},\mathbf{h}$ have strictly positive elements and
 $ \boldsymbol{\nu }\mathbf{h}=1$(see~\cite[p.42]{Asmussen}). Denote
\begin{equation*}%\label{eq22}
C_+=\max_{j\epsilon E'}\frac{1}{h_j}\sup_{x\geq
0}\frac{\overline{F^0_j}(x)}{\int_x^{\infty }e^{\gamma (y-x)}F^0_j(dy)},
\quad C_-=\min_{j\epsilon
E'}\frac{1}{h_j}\inf_{x\geq 0}\frac{\overline{F^0_j}(x)}{\int_x^{\infty }e^{\gamma
(y-x)}F^0_j(dy)}.
\end{equation*}
\begin{corollary}
If $m_1^0<0$, then for all $i\,\epsilon \,E'$ and all $u\geq 0$
\begin{equation}\label{eq23}C_-h_ie^{-\gamma
u}\leq \psi _i(u)\leq C_+h_ie^{-\gamma u}.
\end{equation}
\end{corollary}
\begin{proof}
See the proof of Theorem 3.11~\cite{Asmussen}.
\end{proof}

\par {\bf{Example.}} Let ${Z}(t)=\left\{ {\xi}(t),x(t)\right\}$
be the process on a Markov chain $x(t)$ with infinitesimal matrix: $
\mathbf{Q}=\left(
\begin{matrix}
-1&1\\
1&-1
\end{matrix}\right).$
\par We assume, that $\chi _{kr}=0;\;k,r=1,2$, and component $\xi (t)$ has
 the next representation:
$$\xi _i(t)=S_i(t)-S'_i(t)=\sum_{k\leq\nu'_i (t)}
\eta _k^i-\sum_{k\leq\nu_i (t)}\xi^{i} _k,\; \text{ if } x(t)=i,\;
i=1,2,$$
where $S_i(t),\,S'_i(t)$ are compound Poisson processes with the
rates $\lambda _i,\lambda'_i=1$ and jumps $\xi^{i} _k,\,\eta _k^i
>0$, correspondingly. Moreover,
\begin{gather*}
\mathrm{P}\left\{\xi _k^i>x\right\}=e^{-c_i x},\quad
\frac{\partial}{\partial x}\mathrm{P}\left\{\eta_k^{i} <x\right\}
={\delta_i}^{2} x\, e^{-\delta_i x},\; x\geq 0,\quad i=1,2;\\
c_1=\frac{1}{3},\,c_2=\frac{1}{2},\,\delta _1=2,\,\delta _2=1.
\end{gather*}
Let's find the distribution of absolute maximum, which defines the
ruin probabilities and the distributions of overshoots for zero
level.
\par Consider auxiliary process  $Z_1(t)=\left\{\xi_1 (t),x(t)\right\}=\left\{-\xi
(t),x(t)\right\}$, with cumulant
$$\boldsymbol{\Psi}_1(-\imath r)=\left(
\begin{matrix}
-\frac{9r^3+34r^2+16r
-4}
 {\left(3r-1\right)\left(r+2\right)^2}&1\\
1&-\frac{6r^3+10r^2-1}{(2r-1)(r+1)^2}
\end{matrix}\right).$$
In our case the stationary distribution is defined by
$\boldsymbol{\pi} =\left(\frac{1}{2},\frac{1}{2}\right)$, then
$m_1^0=1>0$. Accordingly to~\cite[Theorem 3]{Gusak5} $\xi_1 ^-$ has
nondegenerate distribution. Consider the matrix
$$\mathbf{G}(r):=r\boldsymbol{\Psi }_1^{-1}(-\imath r)
\left(\mathbf{C}-r\mathbf{I}\right)^{-1}=\frac{1}{D(r)}\left(
\begin{matrix}
g_{11}(r)&g_{12}(r)\\
g_{21}(r)&g_{22}(r)
\end{matrix}\right),$$
where $D(r)=48\,r^5+263\,r^4+387\,r^3+114\,r^2-51\,r-8,\quad
g_{11}(r)=3(r+2)^2(6r^3+10r^2-1)$,\quad$g_{12}(r)=2(r+1)^2(r+2)^2(3r-1)$,\quad
 $g_{21}(r)=3(r+1)^2(r+2)^2(2r-1)$,
$g_{22}(r)=2(r+1)^2(9r^3+34r^2+16r-4)$. Since equation $D(r)=0$ has
four negative roots:

 $-\rho_1 =-3.25672$,\;
$-\rho _2=-1.59682$,\; $-\rho _3=-0.794382$,\; $-\rho _4=-0.133485$

\noindent and one positive $r_0 =0.30224$, then the elements of
matrix $\mathbf{G}(r)$ we can represent in the next form
$$G_{ij}(r)=C_{ij}^0+\frac{C_{ij}^1}{r+\rho _1}+\frac{C_{ij}^2}{r+\rho _2}
  +\frac{C_{ij}^3}{r+\rho _3}+\frac{C_{ij}^4}{r+\rho _4}+\frac{C_{ij}^5}{r-r_0}.$$
\par Use the projection operation (see~\cite[p.34]{Gusak2}), which for functions
$$\mathbf{G}(r)=\mathbf{C}_0 +\int_{-\infty}^{\infty}
e^{rx}\mathbf{g}(x)dx$$ is defined as follows $$\left [
\mathbf{G}(r)\right ]^-=\int_{-\infty}^{0}e^{rx}\mathbf{g}(x)dx,$$
 then
$$G_{ij}^-(r)=\left[G_{ij}(r)\right]^{-}=\frac{C_{ij}^1}{r+\rho _1}
  +\frac{C_{ij}^2}{r+\rho _2}+\frac{C_{ij}^3}{r+\rho _3}+\frac{C_{ij}^4}{r+\rho _4}.$$
Since
$$\ad{\mathbf{R}}_+=\left(\mathbf{G}^-(0)+\left(\boldsymbol{\Lambda
}-\mathbf{Q}\right)^{-1}\right)^{-1}\mathbf{P}_0=\left(
\begin{matrix}
0.22&0.22\\
0.17&0.17
\end{matrix}\right),$$
then according to~\cite[Theorem 3]{Gusak5}
\begin{equation*}
\begin{split}
\mathrm{E}_i\left[e^{r\xi_1 ^-},\xi_1
^-<0\right]&=\mathrm{E}\left[e^{r\xi_1 ^-},\xi_1
^-<0/x(0)=i\right]=\mathbf{E}\left[e^{r\xi_1 ^-},\xi_1 ^-<0\right]
\cdot\mathbf{e}=\\
&=\left[\mathbf{G}(r)\right]^-\ad{\mathbf{R}}_+ \cdot \;\mathbf{e}
=\frac{A_i^1}{r+\rho _1}+\frac{A_i^2}{r+\rho _2}+\frac{A_i^3}{r+\rho
_3}+\frac{A_i^4}{r+\rho _4},\quad i=1,2;
\end{split}
\end{equation*}
where $\mathbf{e}=(1,1)'$. Inverting the last relation with respect
to $r$, we can determine the distribution of $\xi_1 ^-$ as follows:
$$\mathrm{P}_i\left\{\xi_1 ^-<x\right\}=\mathrm{P}\left\{\xi_1 ^-<x/x(0)=i\right\}
=\sum_{k\leq 4}\frac{A_i^k}{\rho _k}e^{\rho _k x},\quad x<0.
$$
That is, we have the next representations of ruin probabilities
\begin{equation*}
\begin{split}
\psi _1(u)&=\mathrm{P}_1\left\{\xi ^+>u\right\}
=\mathrm{P}_1\left\{\xi_1 ^-<-u\right\}\approx\\
 &\approx-0.04\,e^{-3.26 u}+0.001\,e^{-1.6 u}+0.079\,e^{-0.79 u}+0.75\,e^{-0.13 u},
\end{split}
\end{equation*}
\begin{equation*}
\begin{split}
\psi _2(u)&=\mathrm{P}_2\left\{\xi ^+>u\right\}
=\mathrm{P}_2\left\{\xi_1 ^-<-u\right\}\approx\\
&\approx-0.01\,e^{-3.26 u}-0.016\,e^{-1.6 u}+0.004\,e^{-0.79 u}+0.85\,e^{-0.13 u}.
\end{split}
\end{equation*}
From other side, we can use inequalities~\eqref{eq23}:
\begin{gather*}
  0.665 e^{-0.13 u}\leq \psi _1(u)\leq 0.935 e^{-0.13 u},\\
  0.757 e^{-0.13 u}\leq \psi _2(u)\leq 1.064 e^{-0.13 u}.
\end{gather*}
Moreover, form~\eqref{eq19} as $s\rightarrow 0$ we receive
\begin{equation*}
\begin{split}
  \mathbf{P}\left\{\gamma ^{+}(0)>z,\tau ^+(0)<\infty\right\}\approx&
  \left(\begin{matrix}
    e^{-2 z}(0.48+ 0.86z)&e^{-z}(0.31 + 0.22 z)\\
    e^{-2 z}(0.21 + 0.34 z)&e^{-z}(0.61+ 0.49 z)
  \end{matrix}\right),\\
  \mathbf{P}\left\{\gamma _{+}(0)>z,\tau ^+(0)<\infty\right\}\approx&
  \left(\begin{matrix}
  0.1e^{-2 z}(1+z)&e^{-z}(0.2 + 0.1 z)\\
  0.09e^{-2 z}(1+z)&e^{-z}(0.18 + 0.09 z)
  \end{matrix}\right)+\\
  &+
  \left(\begin{matrix}
  e^{-2.3 z}(0.0016 + 0.002 z)&-e^{-1.3 z}(0.02 + 0.013 z)\\
  -e^{-2.3 z}(0.004 + 0.004 z)&e^{-1.3 z}(0.05 + 0.03z)
  \end{matrix}\right),\\
  \mathbf{P}\left\{\gamma ^{+}_0>z,\tau ^+(0)<\infty\right\}\approx&
  \left(\begin{matrix}
  e^{-2 z}(0.49+ 0.97 z+ 0.2z^2)&e^{-z}(0.3(1+z)+0.1z^2)\\
  e^{-2 z}(0.2+ 0.4 z+ 0.18z^2)&e^{-z}(0.7(1+z)+0.09z^2)
  \end{matrix}\right)+\\
  &+
  \left(\begin{matrix}
  -e^{-2.3 z}(0.005 + 0.01z)&e^{-1.3 z}(0.03 + 0.04z)\\
  e^{-2.3 z}(0.01 + 0.03z)&-e^{-1.3 z}(0.075 + 0.1 z)
  \end{matrix}\right).
\end{split}
\end{equation*}
\par

\end{document}